\documentclass[a4paper,11pt]{amsart}

\usepackage[latin1]{inputenc}
\usepackage[T1]{fontenc}
\usepackage[linktocpage]{hyperref} 
\hypersetup{colorlinks=true,linkcolor=blue,citecolor=blue,filecolor=magenta,urlcolor=blue}
\usepackage{geometry}       
\usepackage[english]{babel}  
\usepackage{amsmath,amsfonts,amssymb,amsthm,mathrsfs}
\usepackage{dsfont}
\usepackage{tikz,tkz-graph,tikz-cd}
	\usetikzlibrary[patterns,calc]
\usepackage{array,multirow,makecell}
\usepackage{multicol}
\usepackage{lmodern}

\usepackage{listings}             
\definecolor{mygreen}{rgb}{0,0.6,0}
\definecolor{mygray}{rgb}{0.5,0.5,0.5}
\definecolor{mymauve}{rgb}{0.58,0,0.82}

\newtheorem{thm}{Theorem}
\newtheorem{propo}[thm]{Proposition}

\newtheorem{cor}[thm]{Corollary}

\theoremstyle{definition}

\newtheorem{pb}[thm]{Problem}

\theoremstyle{remark}
\newtheorem{rmk}[thm]{Remark}

\newcommand{\CC}{\mathds{C}}
\newcommand{\RR}{\mathds{R}}

\newcommand{\ZZ}{\mathds{Z}}

\newcommand{\FF}{\mathds{F}}

\newcommand{\PP}{\mathds{P}}

\newcommand{\BB}{\mathds{B}}
\newcommand{\A}{\mathcal{A}}
\newcommand{\B}{\mathcal{B}}

\renewcommand{\L}{\mathcal{L}}

\renewcommand{\epsilon}{\varepsilon}
\newcommand{\m}{\mathfrak{m}}

\newcommand\PC{\CC\PP}

\newcommand{\PCc}{\Check{\PC}}

\DeclareMathOperator{\Sing}{Sing}
\DeclareMathOperator{\Aut}{Aut}

\DeclareMathOperator{\PGL}{PGL}
\DeclareMathOperator{\gr}{gr}

\begin{document}

\title[Fundamental groups of real arrangements and torsion in LCS quotients]{Fundamental groups of real arrangements and torsion in the lower central series quotients}
\author{Enrique Artal Bartolo}
\address{
Departamento de Matem\'aticas, IUMA\\ 
Universidad de Zaragoza\\ 
C.~Pedro Cerbuna 12\\ 
50009 Zaragoza, Spain
}
\email{artal@unizar.es}

\author{Beno\^it Guerville-Ball\'e}

\author{Juan Viu-Sos}
\address{
Instituto de Ci\^encias Matem\'aticas e de Computa\c c\~ao,
Universidade de S\~ao Paulo,
Avenida Trabalhador Sancarlense, 400 - Centro,
S\~ao Carlos - SP, 13566-590, Brasil
}
\email{benoit.guerville-balle@math.cnrs.fr,jviusos@math.cnrs.fr}


%

\begin{abstract}
	By using computer assistance, we prove that the fundamental group of the complement of a real complexified line arrangement is not determined by its intersection lattice, providing a counter-example for a problem of Falk and Randell. We also deduce that the torsion of the lower central series quotients is not combinatorially determined, which gives a negative answer to a question of Suciu.
\end{abstract}

\maketitle


\section*{Introduction}
	
The \emph{topology} of a line arrangement $\A=\{L_0,\dots,L_n\}$ is the homeomorphism type of the pair formed by $\A$ and the complex projective plane $\PC^2$. It is well-known, since Rybnikov~\cite{Rybnikov}, that the \emph{combinatorics} of $\A$ (or equivalently the underlying matroid or intersection lattice of $\A$) does not determine its topology. 
Before~\cite{GueViu:config} there were three known examples of \emph{Zariski pairs}
of line arrangements, i.e. a pair of intersection lattice-equivalent arrangements with different topologies.  
The first one, by Rybnikov~\cite{Rybnikov,accm:03a}, is a pair of line arrangements
admitting no real equation and distinguished by the fundamental group of the complement $G_\A$;
the last one (with coefficients conjugated in a complex Galois-extension) was first distinguished
with a linking property~\cite{Gue:4tuple} and later  by their fundamental group~\cite{ACGM:ZP}; the remaining one~\cite{ACCM:real_ZP} is the only example which can be realized with all lines defined by real coefficients, and its topology is distinguished by the \emph{braid monodromy}, see~\cite{Chisini,Cheniot,Moishezon} in the context of algebraic plane curves and surfaces, and~\cite{Salvetti1,CohenSuciu} for more details in the case of arrangements. It is worth noticing that, in this last example, we currently do not know if their fundamental groups are isomorphic or not, but it turns out that their profinite completions are
(since their equations are conjugated in a real Galois-extension).

In the survey paper~\cite{FalkRandell2}, the authors suggest the following problem for real line arrangements:
\begin{pb}[Falk-Randell]\label{pb:FR}
	Prove that the underlying matroid of a complexified arrangement determines the homotopy type of the complement, or find a counter-example.
\end{pb}

As it is noted in the same paper, some partial results on this problem are already known. Salvetti~\cite{Salvetti} proved that the homotopy type of the complement of a real line arrangement is determined by its oriented matroid. Later, Jiang and Yau showed that the diffeomorphism type of the complement determines the underlying matroid~\cite{JiangYau2}
(see also \cite{pasq:99}), and that the converse is also true if we add a combinatorial condition~\cite{JiangYau}. Finally, Cordovil~\cite{Cordovil} proved that the combinatorics together with a geometrical condition (on the real picture of the arrangement) also determines $G_\A$.

In the other direction, a finite presentation of the fundamental group of complex algebraic plane curves was determined by Zariski and Van Kampen~\cite{zariski, vanKampen}. As was remarked by Chisini~\cite{Chisini}, the previous presentation involves a more powerful topological invariant of curves: the braid monodromy, which was extensively studied by Moishezon, for example in~\cite{Moishezon,MoiTei}. Later on, Libgober showed that the braid monodromy contains all the information of the homotopy type of the curve complement~\cite{Libgober}. The braid monodromy for real complexified arrangements was extensively studied by Salvetti~\cite{Salvetti1,Salvetti}, Hironaka~\cite{Eriko}, and Cordovil and Fachada~\cite{CordovilFachada,Cordovil}; the real complexified Zariski pair of~\cite{ACCM:real_ZP} was distinguished via a
non-generic braid monodromy. The latter implies, due to~\cite{car:xx}, that the braid monodromy is not determined by the combinatorics, answering a question of Cohen and Suciu~\cite[Sec.~1.3]{CohenSuciu}.

%

A natural approach to find a counter-example for Problem~\ref{pb:FR} is then the study of $G_\A$ and its invariants. In~\cite{Suciu}, Suciu remarks the presence of torsion elements in the lower central series quotients of such fundamental groups for complex arrangements. In this way, he queries in the general case:
\begin{pb}[Suciu]\label{pb:Suciu}
	Is the torsion in the quotients of the lower central series combinatorially determined?
\end{pb}

In the present paper, we give a negative answer to above problems exhibiting an explicit example. Indeed, considering one of the real complexified Zariski pairs recently constructed by the last two named authors in~\cite{GueViu:config}, we prove in Corollary~\ref{cor:fondagrp} that the fundamental groups of their complements are not isomorphic. More precisely, in Theorem~\ref{thm:main}, we prove that the $4^\textsc{th}$ and the $5^\textsc{th}$~quotient groups of the lower central series of these fundamental groups differ by a 2-torsion element 
(the computations, made using \texttt{GAP}~\cite{GAP4}, mainly the package~\texttt{nq}~\cite{nq}, are described in Appendix~\ref{sec:appendix}).
This provides a negative answer to Problem~\ref{pb:Suciu}, and thus to Problem~\ref{pb:FR} too.\\

\noindent {\bf Conventions.} For any elements $x,y$ of a group $G$, we denote
$x^y=y^{-1}xy$, $y*x=yxy^{-1}$, as well as $[x,y]=xyx^{-1}y^{-1}$. 
If $L_{i_1}, \ldots, L_{i_k}$ are different lines in $\PC^2$ having a common intersection point, this one is denoted by $P_{i_1,\ldots,i_k}$.\\


\noindent\textbf{Acknowledgments.} 
The authors would like to thank the referee for its valuable comments and suggestions which helped to improve the manuscript. The second and the third named authors are supported by postdoctoral grants \#2017/15369-0 and \#2016/14580-7, respectively, by \emph{Funda\c{c}\~ao de Amparo \`a Pesquisa do Estado de S\~ao Paulo (FAPESP)}. The first and third named authors are partially supported by MTM2016-76868-C2-2-P

\numberwithin{thm}{section}

\bigskip
\section{Fundamental group of the complement and lower central series quotients}\label{sec:funda}
	
A \emph{line arrangement} $\A=\{L_0,\cdots,L_n\}$ is a finite collection of distinct lines in the complex projective plane $\PC^2$. If there exists a system of coordinates of $\PC^2$ such that each line of $\A$ is defined by an $\RR$-linear form, then the arrangement is called \emph{real complexified}.
Denote by $\Sing(\A)=\{L\cap L'\mid L,L'\in\A,\ L\neq L'\}$ the set of singular points of the underlying projective variety $\bigcup\A=\bigcup_{L\in\A}L$.
The \emph{combinatorics} of an arrangement $\A$ is described by its \emph{intersection lattice} $\L(\A)=\{\emptyset\neq \bigcap_{L\in\B}L \mid \B\subset\A\}$ (or equivalently, by its underlying matroid), i.e. the set of singular points and lines of the arrangement as well as the whole $\CC\PP^2$, ordered by reverse inclusion and with rank function given by codimension. The \emph{complement} of $\A$ is $M(\A)=\PC^2\setminus \bigcup\A$, i.e. a smooth quasiprojective manifold. While the complement of any hypersurface of $\PC^N$ is affine, in our case this result becomes easier since $\PC^2\setminus L_0\cong\CC^2$.


As a first step for our purposes we need to compute $G_\A=\pi_1(M(\A))$, the fundamental group of the complement
of some real complexified arrangements. In such case, Randell~\cite{Randell} gives an algorithm to compute a finite presentation of $G_\A$ from the real picture of the arrangement,
see also~\cite{Salvetti1}. This has been generalized by Arvola in~\cite{Arvola} to any complex line arrangement.

\subsection{Computation of the fundamental group}\label{sec:group}
\mbox{}

From now on, we will assume that $\A$ is a complexified real arrangement. Let us recall how to compute the presentation of $G_\A$ given by Randell. We assume that $L_0$ is the line $z=0$ and is considered as the line at infinity. Let $\A^\RR=\A\cap\RR^2$ be the real picture of $\A$. We also assume that no line of $\A$ is of the form $x=\alpha$ with $\alpha\in\RR$. Assign to each line of $\A^\RR$ a meridian contained in a line $x=\beta$ containing no multiple point.
Let denote the meridian of $L_i$ by $\m_i$.

Reading the picture $\A^\RR$ from left to the right, we assign to each smooth part in $\A^\RR$ (i.e. the segments bounded by the singularities) a conjugate of the meridian of the associated line. This process is described in Figure~\ref{fig:Arvola}.

\begin{figure}[ht!]
	\begin{tikzpicture}
	\begin{scope}[xscale=4]
		\draw (0,0) -- (1,3);
		\draw (0,1) -- (1,2);
		\draw (0,2) -- (1,1);
		\draw (0,3) -- (1,0);
		\node[left] at (0,0) {$\omega_\ell$};
		\node[left] at (0,1) {$\omega_{\ell-1}$};
		\node[left] at (-0.05,1.55) {$\vdots$};
		\node[left] at (0,2) {$\omega_2$};
		\node[left] at (0,3) {$\omega_1$};
		\node[right] at (1,0) {$\omega_1$};
		\node[right] at (1,1) {$\omega_1*\omega_2$};
		\node[right] at (1.1,1.55) {$\vdots$};
		\node[right] at (1,2.1) {$(\omega_1\cdots\omega_{\ell-2})*\omega_{\ell-1} \equiv \omega_{\ell-1}^{\omega_\ell}$};
		\node[right] at (1,3) {$(\omega_1\cdots \omega_{\ell-1})*\omega_\ell\equiv \omega_\ell $};
		\node at (-1,0) {};
	\end{scope}
\end{tikzpicture}
	\caption{Assignment of meridians at each singularity\label{fig:Arvola}.}
\end{figure}
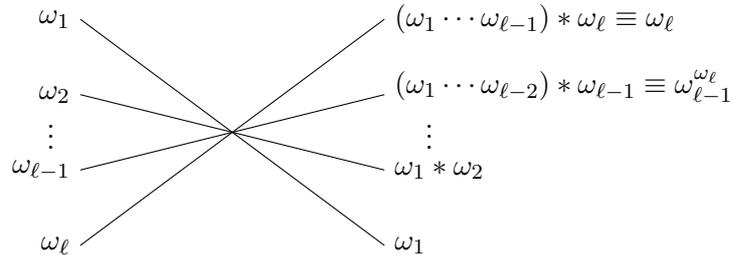

Finally, to each singular point $P\in\Sing(\A^\RR)$ with input elements $\omega_1,\ldots,\omega_\ell$ associated to the left-hand segments (as in Figure~\ref{fig:Arvola}), we assign the set of relations
\[
	R_P=\{\omega_1\cdots\omega_\ell=\omega_{\sigma(1)}\cdots \omega_{\sigma(\ell)} \mid \sigma\ \text{a cyclic permutation of $\ell$ elements}\}
\]

\begin{rmk}\label{rmk:vertical1}
This method works even if there are vertical lines; by a small rotation, e.g. counterclockwise, one can assume that
the vertical line is the first one. How it works for double or triple points is shown in
Figure~\ref{fig:vertical}.
\begin{figure}[ht!]
	\begin{tikzpicture}
	\begin{scope}[xscale=1.5,yscale=.5]
		\draw (0,1.5) -- (1,1.5);
		\draw (0.5,3) -- (0.5,0);
		\node[below] at (0.5,0) {$\omega_1$};
		\node[above] at (0.5,3) {$\omega_{1}$};
		\node[left] at (-0.05,1.55) {$\omega_2$};
		\node[right] at (1.1,1.55) {$\omega_2$};
	\end{scope}

	\begin{scope}[xscale=2,yscale=.5,xshift=3cm]
		\draw (0,0) -- (1,3);
		\draw (0.5,3) -- (.5,0);
		\draw (0,3) -- (1,0);
		\node[left] at (0,0) {$\omega_3$};
		\node[above] at (0.5,3) {$\omega_1$};
		\node[left] at (0,3) {$\omega_2$};
		\node[below] at (0.5,0) {$\omega_1$};
		\node[right] at (1,0) {$\omega_1*\omega_2\equiv\omega_2^{\omega_3}$};
		\node[right] at (1,3) {$(\omega_1\omega_2)*\omega_3\equiv \omega_3 $};
		\node at (-4.5,0) {};
	\end{scope}
\end{tikzpicture}
	\caption{Assignments of meridians with vertical lines\label{fig:vertical}.}
\end{figure}
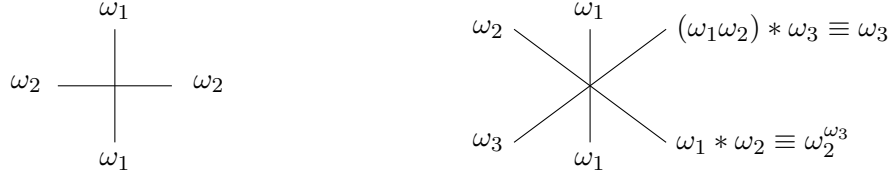
\end{rmk}

\begin{thm}[Randell~\cite{Randell}]\label{thm:arvola}
	The fundamental group of the complement of $\A$ admits the following presentation:
	\begin{equation*}
		G_\A\simeq\left\langle\ \m_1,\dots,\m_n\ \left|\ \bigcup_{P\in\Sing(\A^\RR)} R_P\ \right.\right\rangle.
	\end{equation*}
\end{thm}

Under the same hypothesis as above, another method to obtain a finite presentation for $G_\A$ is obtained by Zariski-Van Kampen~\cite{zariski,vanKampen} using the braid monodromy. Consider the affine arrangement $\A^\text{aff}=\A\cap\CC^2$ and a generic linear projection with respect to the affine arrangement $\pi: \CC^2\to\CC$ and the finite set $\pi\left(\Sing(A^\text{aff})\right)=\{p_1,\ldots,p_s\}$ ordered by their real part in $\CC$. Let~$\BB_n$ be the \emph{full braid group} given by the standard Artin presentation:
\[
	\BB_n=\left\langle\sigma_1,\ldots,\sigma_{n-1} \left|\ \begin{array}{c}
		\left[\sigma_i,\sigma_j\right]=1,\ |i-j|\geq 2,\\ \sigma_i\sigma_{i+1}\sigma_i=\sigma_{i+1}\sigma_i\sigma_{i+1},\ i=1,\ldots,n-2
	\end{array}\right.\right\rangle,
\]
and consider the \emph{pure braid group of $n$ strings}~$\PP_n$ as a subgroup of $\BB_n$ given by the short exact sequence
\[
	1\to\PP_n\to\BB_n\to\Sigma_n\to 1
\]
where $\BB_n\to\Sigma_n$ assigns to any braid the corresponding permutation in the $n$-symmetric group $\Sigma_n$.

We associate to each $p_i$ a meridian $\gamma_i\in\pi_1(\CC\setminus\{p_1,\ldots,p_s\})\simeq\FF_s$ and a pure braid $\alpha_i\in\PP_n$ from the braided fiber $\pi^{-1}(\gamma_i)$, constructing a basis $\{\gamma_1,\ldots, \gamma_s\}$ of $\FF_s$ and a well-ordered tuple of pure braids $(\alpha_1,\ldots, \alpha_s)$. The \emph{braid monodromy} is the morphism $\rho:\pi_1(\CC^2\setminus\{p_1,\ldots,p_s\})\to\PP_n$ defined by the construction above, and it is well-defined up to conjugations in $\PP_n$. 
Seen as tuples of pure braids, two braid monodromies are said to be \emph{equivalent} if they are equivalent by conjugation in $\PP_n$ and by \emph{Hurwitz moves}, see~\cite{acc:01a}. Using the previous construction and the natural action of $\BB_n$ on the free group $\FF_n=\left\langle x_1,\ldots,x_n\right\rangle$ given by
\[
	x_i^{\sigma_j}=\left\lbrace
	\begin{array}{ll}
		x_{i}x_{i+1}x_{i}^{-1} & , i=j\\
		x_{i-1} & , i=j+1\\
		x_i & , \text{otherwise}
	\end{array}
	\right..
\]
Finally, from any equivalent braid monodromy, we obtain:
\begin{equation*}
	G_\A\simeq\left\langle\ x_1,\ldots,x_n\ \mid x_i^{\alpha_j}=x_i,\ \text{for } i=1,\ldots,n-1\ \text{and }  j=1,\ldots,s\right\rangle.
\end{equation*}
We can reduce this presentation. In fact, these braids can be decomposed as
$\alpha_j=\beta_j^{-1}\Delta^2_{n_j,r_j}\beta_j$, where $\beta_j$ is a braid and $\Delta^2_{n_j,r_j}$
is the full-twist involving $r_j$ strands starting from a suitable $n_j$ ($r_j$ is the multiplicity of
the point associated to $p_j$). Then,
\begin{equation*}
	G_\A\simeq\left\langle\ x_1,\ldots,x_n\ \mid \left(x_i^{\Delta^2_{n_j,r_j}}\right)^{\beta_j}=x_i^{\beta_j},\ \text{for } i=n_j,\ldots,n_j+r_j-2\ \text{and }  j=1,\ldots,s\right\rangle.
\end{equation*}
\begin{rmk}
	The above presentations of $G_\A$ are equivalent by Tietze-I moves as it is proved by Cohen-Suciu~\cite{CohenSuciu}. They have the same number of generators and relators, and the relations are composed, in both presentations, of cyclic commutators involving conjugates of the generators. Thus, we choose the Randell's presentation in order to compute the fundamental group of the complement of our example in a \texttt{GAP} code given in Appendix~\ref{sec:appendix}.
\end{rmk}
\begin{rmk}
	A modified way to obtain the braid monodromy and its associated presentation of $G_\A$ for an affine arrangement with vertical lines can be done using Remark~\ref{rmk:vertical1}, see also~\cite{CohenSuciu}.
\end{rmk}

\subsection{Lower central series quotients}
\mbox{}

The \emph{lower central series} (LCS) of a group $G$ is defined as a descending sequence of normal subgroups
\begin{equation*}
	G = \gamma_1(G) \trianglerighteq \gamma_2(G) \trianglerighteq \cdots \trianglerighteq \gamma_k(G) \trianglerighteq \cdots,
\end{equation*}
such that each $\gamma_{k+1}(G)$ is the commutator subgroup $[\gamma_k(G),G]$ of $G$. The \emph{$k^\textsc{th}$ lower central quotient} of this series is defined as the group $\gr_k(G)=\gamma_k(G)/\gamma_{k+1}(G)$. Note that $\gr_k(G)$ is abelian since $[\gamma_k(G),\gamma_k(G)]\leq \gamma_{k+1}(G)$. 
In addition, if $G$ is finitely generated, also 
$\gr_k(G)$ 
is finitely generated. 
Given $x_1,\dots,x_k\in G$, we denote inductively $[x_1]=x_1$ and
\[
[x_1,\dots,x_k]=[[x_1,,\dots,x_{k-1}],x_k]\in\gamma_k(G).
\]
It is not difficult to prove that the class of $[x_1,\dots,x_k]$ in $\gr_k(G)$ is determined
by the conjugacy class of $x_1,\dots,x_k$ and that the natural action of $G$ on $\gr_k(G)$
(induced by conjugation) is trivial. These facts imply that if $G$ is finitely
generated, this is the case for $\gr_k(G)$, too. Thus, $\gr_k(G)$ is determined by its rank, noted $\phi_k(G)$, and its torsion.

\begin{rmk}
  If $G$ is a finitely presented group, then the GAP package \texttt{nq} can be used to compute the lower central series quotients of $G$.
\end{rmk}

In the case of a group $G_\A$ associated to a line arrangement $\A$, the possible dependency on the combinatorics of the invariants of the LCS is a classical research topic. %
In~\cite{Rybnikov}, Rybnikov gives a sketch of a proof of the combinatorial determination of the second nilpotent group $G_\A/\gamma_3(G_\A)$, lately formally proved by Matei and Suciu~\cite{MateiSuciu}. %

Concerning the cohomology algebra of $M(\A)$, Orlik and Solomon~\cite{OrlikSolomon} showed that it is determined by $L(\A)$ by constructing a combinatorial algebra which is isomorphic to $H^\bullet(M(\A))$,
previously computed by 
Arnol'd and Brieskorn~\cite{Brieskorn} in some outstanding cases. %
Using Sullivan 1-minimal models, Falk proves that also $\phi_k(G_\A)$ is combinatorially determined~\cite{Falk:minimal}. 
He also obtains with Randell~\cite{FalkRandell:LCS}, in the particular case of fiber-type arrangements, the \emph{LCS formula}:
\begin{equation*}
	\prod\limits_{k\geq 1}(1-t^k)^{\phi_k(G_\A)}=P(M(\A),-t),
\end{equation*}
where $P(M(\A),t)$ is the Poincar\'e polynomial of the complement (see~\cite[Sec.~2]{OrlikTerao92}), which is known to be combinatorial. Nevertheless, there exist examples of arrangements for which the LCS formula fails~\cite{FalkRandell,Suciu}, and no general explicit formula is known for $\phi_k(G_\A)$ (even for $\phi_3(G_\A)$).

\bigskip
\section{An explicit example of 13 lines}\label{sec:ZP}
	Let $\L$ be a lattice describing a particular combinatorics of a line arrangement with $n$ lines. We define the 
\emph{moduli space of ordered realizations} of $\L$, noted $\Sigma_\L^{\text{ord}}$, as the set of all line arrangements (expressed as the collection of the associated linear forms) whose combinatorics are lattice-isomorphic to $\L$, i.e.
\[
	\Sigma_\L^{\text{ord}}=\{\A\in (\PCc^2)^n \mid \L(\A)\sim\L\}/\PGL_3(\CC).
\]
It is worth noticing that $\Sigma_\L^{\text{ord}}$ is a constructible set in $(\PC^2)^n$ under the action of $\PGL_3(\CC)$. The \emph{moduli space of realizations} is defined as $\Sigma_\L=\Sigma_\L^{\text{ord}}/\Aut(\L)$, where $\Aut(\L)$ is the automorphism group of the combinatorics $\L$. The study of the topology of moduli spaces is fundamental in order to obtain Zariski pairs, since Randell~\cite{Randell:lattice} proved that two 
line arrangements lying in the same path-component of $\Sigma_\L$ have isomorphic embedded topology.

\subsection{Definition of the example}
\mbox{}

In~\cite{GueViu:config}, several examples of real complexified Zariski pairs are given. In particular, the modulli space of realizations $\Sigma$ of one of them is studied in~\cite[Sec.~4]{GueViu:config}. This particular pair possesses 13 lines, 11 points of multiplicity 3, 2 of multiplicity 5 and the following combinatorics:
\begin{equation*}
	\begin{array}{c}
		\left\{\ 
			\big\{L_1,L_4,L_6,L_9,L_{13}\big\},\ \big\{L_1,L_5,L_7\big\},\ \big\{L_1,L_8,L_{10}\big\},\ \big\{L_1,L_{11},L_{12}\big\},  
			\right. \\[0.5em]
			\quad \big\{L_2,L_4,L_7,L_{10},L_{12}\big\},\ \big\{L_2,L_5,L_6\big\},\ \big\{L_2,L_8,L_9\big\},\ \big\{L_2,L_{11},L_{13}\big\},  \\[0.5em]
			\left. \qquad
			\big\{L_3,L_4,L_5\big\},\ \big\{L_3,L_6,L_8\big\},\ \big\{L_3,L_7,L_{11}\big\},\ \big\{L_3,L_9,L_{10}\big\},\ \big\{L_3,L_{12},L_{13}\big\}\
		\right\},
	\ \end{array}
\end{equation*}
where each subset describes an incidence relation between lines; pairs which do not appear
correspond to double points. It is proven in~\cite[Prop.~4.5]{GueViu:config} that $\Sigma$ decomposes as topological space on two connected components $\Sigma=\Sigma_0\sqcup\Sigma_1$, each of them being isomorphic to a punctured complex affine line, where any pair of arrangements $\A\in\Sigma_0$ and $\A'\in\Sigma_1$ forms a Zariski pair. Furthermore, the arrangements lying on each connected component can be differentiated by particular geometrical properties about the position of their singular points and lines with respect to curves of low degree.

\begin{propo}[{\cite[Thm.~3.5,~Prop.~4.6]{GueViu:config}}]
	The moduli space $\Sigma$ is composed of two connected components $\Sigma_0$ and $\Sigma_1$. Moreover, for any $\A\in\Sigma$, the following are equivalent:
	\begin{enumerate}
		\item $\A\in\Sigma_0$.
		\item The six lines $L_6, L_7, L_8, L_{10}, L_{11}, L_{13}$ are tangent to a smooth conic.
		\item The six triple points $P_{1,8,10}$, $P_{1,11,12}$, $P_{2,8,9}$, $P_{2,11,13}$, $P_{3,9,10}$ and $P_{3,12,13}$ are contained in a smooth conic.
		\item The three triple points $P_{1,11,12}$, $P_{2,8,9}$ and $P_{3,4,5}$ are aligned.
		\item For any $\A'\in\Sigma_1$, we have $(\PC^2,\A)\not\simeq(\PC^2,\A')$.
	\end{enumerate}
\end{propo}

\begin{rmk}
	There exist 3 additional alignments and 11 smooth conics containing 6 singular points (of multiplicity at least 3);
	any of these properties characterizes as well the arrangements belonging to the connected component $\Sigma_0$. 
\end{rmk}


As a consequence of the previous proposition, any pair of arrangements belonging to different connected components forms a Zariski pair. Let consider $\A^+\in\Sigma_0$ and $\A^-\in\Sigma_1$ defined by the following equations:
\begin{gather*}
		L_1: x=0,\quad
		L_2: y=0,\quad
		L_3: x+y-z=0,\quad
		L_4: z=0,
\\
L_5: 3x+3y+z=0,\quad
		L_6: 3x+z=0,\quad
		L_7: 3y+z=0,
\\
L_8: 2x-y+2 z=0,\quad
		L_9: x+z=0,\quad
		L_{10}: y-2z=0,
\\
L_{11}: (-1\pm 1/2 )x+(-1\pm2)y+z=0,
\\
L_{12}: (-1\pm2)y+z=0,\quad
		L_{13}: (-2\pm1)x+2z=0.
 \end{gather*}
Considering $L_4$ as the line at infinity, the affine pictures of $\A^+$ and $\A^-$ are given in Figure~\ref{fig:ZP-Ap} and Figure~\ref{fig:ZP-Am}, respectively.

\subsection{Main results}
\mbox{}

Consider $G_{\A^+}$ and $G_{\A^-}$ being  the fundamental groups of the complements $M(\A^+)$ and $M(\A^-)$, respectively. Using the algorithm given in Section~\ref{sec:funda} and Figures~\ref{fig:ZP-Ap} and~\ref{fig:ZP-Am}, we can compute that in $G_{\A^+}$ and $G_{\A^-}$, the torsion of the $4^\textsc{th}$ and $5^\textsc{th}$ LCS quotients differs. 

%

\begin{thm}\label{thm:main}
	Let $\phi_k$ be the rank $\phi_k(G_{\A^+})=\phi_k(G_{\A^-})$, for any $k>0$. We have:
	\begin{enumerate}
	\item For $k\leq3$, $\gr_k(G_{\A^+})\simeq\gr_k(G_{\A^-})$ are free abelian groups of rank $\phi_k$.
	\item The groups $\gr_4(G_{\A^+})$ and $\gr_5(G_{\A^+})$ contain a non-trivial 2-torsion element, while $\gr_4(G_{\A^-})$ and $\gr_5(G_{\A^-})$ are abelian torsion-free of ranks $\phi_4$ and $\phi_5$ respectively.
	\end{enumerate}	 
\end{thm}

\begin{proof}
  Using the \texttt{GAP} code described in Appendix~\ref{sec:appendix}, we obtain the following primary decompositions of the lower central series quotients:
	\begin{enumerate}
	\item For $k\leq3$, we have $\gr_k(G_{\A^+})\simeq\gr_k(G_{\A^-})\simeq\ZZ^{\phi_k}$ with $(\phi_k)_{k=1}^3=(12, 23, 76)$,
	\item $\gr_4(G_{\A^+})\simeq\ZZ^{211}\oplus\ZZ_2$ and $\gr_4(G_{\A^-})\simeq\ZZ^{211}$,
	\item $\gr_5(G_{\A^+})\simeq\ZZ^{660}\oplus\ZZ_2$ and $\gr_5(G_{\A^-})\simeq\ZZ^{660}$.\qedhere
	\end{enumerate}
\end{proof}

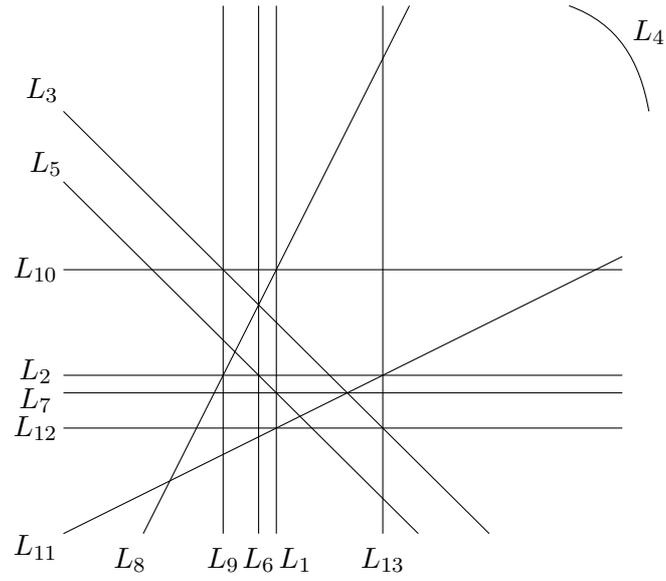
\begin{figure}[ht!]
	\begin{tikzpicture}[scale=0.7]
\def\colorV{black}
\def\colorS{black}
\def\colorSa{black}
\def\colorSb{black}

\coordinate (P1) at (0, -3);
\coordinate (Q1) at (0, 7);
\coordinate (P2) at (-4, 0);
\coordinate (Q2) at (13/2, 0);
\coordinate (P3) at (-4, 5);
\coordinate (Q3) at (4, -3);
\coordinate (P4) at (-4, 11/3);
\coordinate (Q4) at (8/3, -3);
\coordinate (P5) at (-1/3, -3);
\coordinate (Q5) at (-1/3, 7);
\coordinate (P6) at (-4, -1/3);
\coordinate (Q6) at (13/2, -1/3);
\coordinate (P7) at (-5/2, -3);
\coordinate (Q7) at (5/2, 7);
\coordinate (P8) at (-1, -3);
\coordinate (Q8) at (-1, 7);
\coordinate (P9) at (-4, 2);
\coordinate (Q9) at (13/2, 2);
\coordinate (P10) at (-4, -3);
\coordinate (Q10) at (13/2, 9/4);
\coordinate (P11) at (-4, -1);
\coordinate (Q11) at (13/2, -1);
\coordinate (P12) at (2, -3);
\coordinate (Q12) at (2, 7);
\draw[color=\colorV] (P1)--(Q1) node[pos=-.05, xshift=0.25cm] {$L_{1}$};
\draw[color=\colorV] (P2)--(Q2) node[pos=-.05, yshift=0.05cm] {$L_{2}$};
\draw[color=\colorV] (P3)--(Q3) node[pos=-.05] {$L_{3}$};
\draw[color=\colorS] (P4)--(Q4) node[pos=-.05] {$L_{5}$};
\draw[color=\colorS] (P5)--(Q5) node[pos=-.05] {$L_{6}$};
\draw[color=\colorS] (P6)--(Q6) node[pos=-.05, yshift=-0.1cm] {$L_{7}$};
\draw[color=\colorSa] (P7)--(Q7) node[pos=-.05] {$L_{8}$};
\draw[color=\colorSa] (P8)--(Q8) node[pos=-.05] {$L_{9}$};
\draw[color=\colorSa] (P9)--(Q9) node[pos=-.05]
{$L_{10}$};
\draw[color=\colorSb] (P10)--(Q10) node[pos=-.05]
{$L_{11}$};
\draw[color=\colorSb] (P11)--(Q11) node[pos=-.05]
{$L_{12}$};
\draw[color=\colorSb] (P12)--(Q12) node[pos=-.05]
{$L_{13}$};

\draw[color=\colorS] (5.5,7) to[out=-20,in=100] (7,5);
\node[text=\colorS] at (7,6.5) {$L_4$};

\end{tikzpicture}
	\caption{Arrangement $\A^+\in\Sigma_0$\label{fig:ZP-Ap}.}
\end{figure}

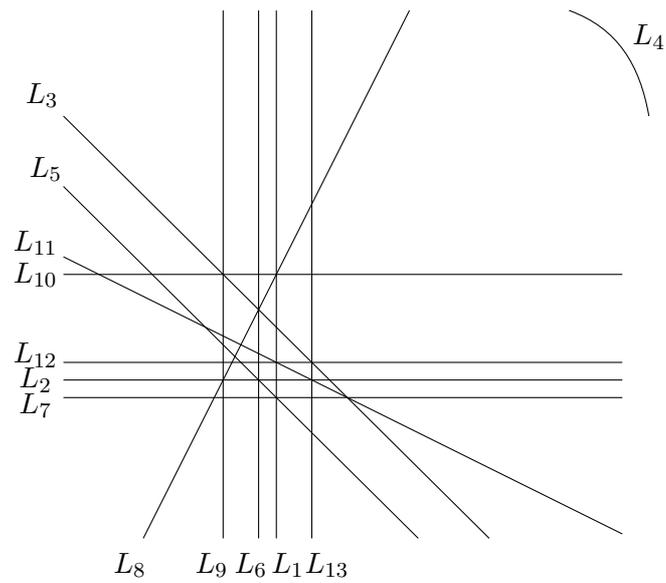
\begin{figure}[ht!]
	\begin{tikzpicture}[scale=0.7]
\def\colorV{black}
\def\colorS{black}
\def\colorSa{black}
\def\colorSb{black}

\coordinate (P1) at (0, -3);
\coordinate (Q1) at (0, 7);
\coordinate (P2) at (-4, 0);
\coordinate (Q2) at (13/2, 0);
\coordinate (P3) at (-4, 5);
\coordinate (Q3) at (4, -3);
\coordinate (P4) at (-4, 11/3);
\coordinate (Q4) at (8/3, -3);
\coordinate (P5) at (-1/3, -3);
\coordinate (Q5) at (-1/3, 7);
\coordinate (P6) at (-4, -1/3);
\coordinate (Q6) at (13/2, -1/3);
\coordinate (P7) at (-5/2, -3);
\coordinate (Q7) at (5/2, 7);
\coordinate (P8) at (-1, -3);
\coordinate (Q8) at (-1, 7);
\coordinate (P9) at (-4, 2);
\coordinate (Q9) at (13/2, 2);
\coordinate (P10) at (-4, 7/3);
\coordinate (Q10) at (13/2, -35/12);
\coordinate (P11) at (-4, 1/3);
\coordinate (Q11) at (13/2, 1/3);
\coordinate (P12) at (2/3, -3);
\coordinate (Q12) at (2/3, 7);
\draw[color=\colorV] (P1)--(Q1) node[pos=-.05, xshift=0.15cm] {$L_{1}$};
\draw[color=\colorV] (P2)--(Q2) node[pos=-.05] {$L_{2}$};
\draw[color=\colorV] (P3)--(Q3) node[pos=-.05] {$L_{3}$};
\draw[color=\colorS] (P4)--(Q4) node[pos=-.05] {$L_{5}$};
\draw[color=\colorS] (P5)--(Q5) node[pos=-.05, xshift=-0.1cm] {$L_{6}$};
\draw[color=\colorS] (P6)--(Q6) node[pos=-.05, yshift=-0.1cm] {$L_{7}$};
\draw[color=\colorSa] (P7)--(Q7) node[pos=-.05] {$L_{8}$};
\draw[color=\colorSa] (P8)--(Q8) node[pos=-.05, xshift=-0.15cm] {$L_{9}$};
\draw[color=\colorSa] (P9)--(Q9) node[pos=-.05]
{$L_{10}$};
\draw[color=\colorSb] (P10)--(Q10) node[pos=-.05]
{$L_{11}$};
\draw[color=\colorSb] (P11)--(Q11) node[pos=-.05, yshift=0.1cm]
{$L_{12}$};
\draw[color=\colorSb] (P12)--(Q12) node[pos=-.05, xshift=0.2cm]
{$L_{13}$};

\draw[color=\colorS] (5.5,7) to[out=-20,in=100] (7,5);
\node[text=\colorS] at (7,6.5) {$L_4$};
\end{tikzpicture}
	\caption{Arrangement $\A^-\in\Sigma_1$\label{fig:ZP-Am}.}
\end{figure}

\begin{rmk}
	We can show that the commutator of meridians $[[[\m_1,\m_5],\m_2], \m_3]$ is a representative of the 2-torsion element in $\gr_4(G_{\A^+})$ (using \texttt{GAP}), while it corresponds to the identity in $\gr_4(G_{\A^-})$.
\end{rmk}

\begin{cor}\label{cor:torsion}
	The torsion in quotients of the LCS of $G_\A$ is not determined by the intersection lattice $\L(\A)$.
\end{cor}

\begin{rmk}
	It can be verified that the LCS formula fails for these arrangements.
\end{rmk}

\begin{cor}\label{cor:fondagrp}
  The fundamental group of the complement of a real complexified arrangement $\A$ is not determined by its intersection lattice $\L(\A)$.
\end{cor}
%

\begin{rmk}
	The previous results can be verified by using a presentation coming from the braid monodromy, as it is explained in Section~\ref{sec:group}. Note that Theorem~\ref{thm:main} implies in particular that $\A^+$ and $\A^-$ have non-equivalent braid monodromies.
\end{rmk}



	
\bigskip
\section{Appendix}\label{sec:appendix}
	
In this appendix, we provide a \texttt{GAP} code
which computes finite presentations of the fundamental groups of the arrangements defined above, by using Randell's method (Theorem~\ref{thm:arvola}). From this, we compute the respective lower central series quotients mentioned in Theorem~\ref{thm:main}.

\begin{lstlisting}
# We load the GAP package nq, and define the group *$$m:=\mathds{F}_{13}$$*
LoadPackage('nq')
m:=FreeGroup( 13 );

# We define the function which produces the commutation relations from an ordered list of elements  representing each crossing.
comm:=function(L,f)
	local l,C,c,i,j;
	l:=Length(L);
	C:=[ ];
	for i in [1..l-1] do
		c:=Identity(f);
		for j in [1..l] do
			c:=c*L[j];
		od;
		for j in [i..l+i-1] do
			c:=c*L[RemInt(2*l-j-1,l)+1]^(-1);
		od;
		Append(C,[c]);
	od;
	return C;
end;

# We define a function computing the list of relations taking as input an ordered list of crossings, following Randell's result and our conventions.
relations:=function(L,f,L_infinity)
	local gen, rel, m_in, m_out, r, ind, rloc, l, i, j, conj;
	gen:=GeneratorsOfGroup(f);
	rel:=[gen[L_infinity]];
	m_in:=List(gen,x->x);
	m_out:=List(gen,x->x);
	for r in L do
		ind:=List(r,x->Position(gen,x));
		rloc:=List(ind,x->m_in[x]);
		Append(rel,comm(rloc,f));
		l:=Length(ind);
		for i in [1..l] do
			conj:=m_in[ind[i]];
			for j in [1..i-1] do
				conj:=m_in[ind[i-j]]*conj*m_in[ind[i-j]]^(-1);
			od;
			m_out[ind[i]]:=conj;
		od;
	    m_in:=List(m_out,x->x);
	od;
	return rel;
end;

# List of crossings of *$$\A^+$$*, finite presentation of its fundamental group and computation of the lower central series.
cross1:=[[m.5,m.10],[m.11,m.8],[m.12,m.8],[m.7,m.8],[m.9,m.3,m.10],[m.9,m.5],[m.9,m.2,m.8],[m.9,m.7],[m.9,m.12],[m.9,m.11],[m.5,m.8],[m.6,m.10],[m.6,m.3,m.8],[m.6,m.5,m.2],[m.6,m.7],[m.6,m.12],[m.6,m.11],[m.1,m.10,m.8],[m.1,m.3],[m.1,m.2],[m.1,m.5,m.7],[m.1,m.12,m.11],[m.5,m.11],[m.5,m.12],[m.3,m.2],[m.3,m.7,m.11],[m.13,m.8],[m.13,m.10],[m.13,m.2,m.11],[m.13,m.7],[m.13,m.3,m.12],[m.13,m.5],[m.10,m.11]];;
rels1 := relations( cross1, m, 4 );;
g1 := m / rels1;
L1 := LowerCentralFactors( g1, 5 );

# For  *$$\gr_i(G_{\A^+})$$*, with *$$i\in\{1,\dots,5\}$$*, we compute: the number of generators, the rank (which is combinatorial) and the number of 2-torsion generators.
List( [1..5], i -> Length( L1[i] ) );
List( [1..5], i -> Number( L1[i], x -> x=0 ) );
List( [1..5], i -> Number( L1[i], x -> x=2 ) );

# List of crossings of *$$\A^-$$*, finite presentation of its fundamental group and computation of the lower central serie.
cross2:=[[m.11,m.10],[m.5,m.10],[m.5,m.11],[m.7,m.8],[m.9,m.3,m.10],[m.9,m.11],[m.9,m.5],[m.9,m.12],[m.9,m.2,m.8],[m.9,m.7],[m.12,m.8],[m.5,m.8],[m.5,m.12],[m.11,m.8],[m.6,m.10],[m.6,m.3,m.8],[m.6,m.11],[m.6,m.12],[m.6,m.5,m.2],[m.6,m.7],[m.1,m.10,m.8],[m.1,m.3],[m.1,m.11,m.12],[m.1,m.2],[m.1,m.5,m.7],[m.13,m.8],[m.13,m.10],[m.13,m.3,m.12],[m.13,m.11,m.2],[m.13,m.7],[m.13,m.5],[m.3,m.2],[m.3,m.11,m.7]];;
rels2 := relations( cross2, m,4 );;
g2 := m / rels2;
L2 := LowerCentralFactors( g2, 5 );

# For  *$$\gr_i(G_{\A^-})$$*, with *$$i\in\{1,\dots,5\}$$*, we compute: the number of generators, the rank (which is combinatorial) and the number of 2-torsion generators.
List( [1..5], i -> Length( L2[i] ) );
List( [1..5], i -> Number( L2[i], x -> x=0 ) );
List( [1..5], i -> Number( L2[i], x -> x=2 ) );
\end{lstlisting}


\begin{thebibliography}{ACGM17}

\bibitem[ACC03]{acc:01a}
E.~Artal, J.~Carmona, and J.I. Cogolludo{-Agust{\'i}n}, \emph{Braid monodromy
  and topology of plane curves}, Duke Math. J. \textbf{118} (2003), no.~2,
  261--278.

\bibitem[ACCM05]{ACCM:real_ZP}
E.~Artal, J.~Carmona, J.I. Cogolludo{-Agust{\'i}n}, and M.~Marco,
  \emph{Topology and combinatorics of real line arrangements}, Compos. Math.
  \textbf{141} (2005), no.~6, 1578--1588.

\bibitem[ACCM07]{accm:03a}
\bysame, \emph{Invariants of combinatorial line arrangements and {R}ybnikov's
  example}, Singularity theory and its applications, Adv. Stud. Pure Math.,
  vol.~43, Math. Soc. Japan, Tokyo, 2007, pp.~1--34.

\bibitem[ACGM17]{ACGM:ZP}
E.~Artal, J.I. Cogolludo{-Agust{\'i}n}, B.~Guerville{-Ball{\'e}}, and M.~Marco,
  \emph{An arithmetic {Z}ariski pair of line arrangements with non-isomorphic
  fundamental group}, Rev. R. Acad. Cienc. Exactas F\'\i s. Nat. Ser. A Math.
  RACSAM \textbf{111} (2017), no.~2, 377--402.

\bibitem[Arv92]{Arvola}
W.A. Arvola, \emph{The fundamental group of the complement of an arrangement of
  complex hyperplanes}, Topology \textbf{31} (1992), no.~4, 757--765.

\bibitem[Bri73]{Brieskorn}
E.V. Brieskorn, \emph{Sur les groupes de tresses {\rm[}d'apr\`es {V}. {I}.
  {A}rnol'd{\rm]}}, S\'eminaire Bourbaki, 24\`eme ann\'ee (1971/1972), Exp. No.
  401, Springer, Berlin, 1973, pp.~21--44. Lecture Notes in Math., Vol. 317.

\bibitem[Car03]{car:xx}
J.~Carmona, \emph{Monodrom{\'i}a de trenzas de curvas algebraicas planas},
  Ph.D. thesis, Universidad de Zaragoza, 2003.

\bibitem[CF95]{CordovilFachada}
R.~Cordovil and J.L. Fachada, \emph{Braid monodromy groups of wiring diagrams},
  Boll. Un. Mat. Ital. B (7) \textbf{9} (1995), no.~2, 399--416.

\bibitem[Che73]{Cheniot}
D.~Cheniot, \emph{Une d\'emonstration du th\'eor\`eme de {Z}ariski sur les
  sections hyperplanes d'une hypersurface projective et du th\'eor\`eme de
  {V}an {K}ampen sur le groupe fondamental du compl\'ementaire d'une courbe
  projective plane}, Compositio Math. \textbf{27} (1973), 141--158.
  \MR{0366922}

\bibitem[{Chi}33]{Chisini}
O.~{Chisini}, \emph{{Una suggestiva rappresentazione reale per le curve
  algebriche piane.}}, {Ist. Lombardo, Rend., II. Ser.} \textbf{66} (1933),
  1141--1155 (Italian).

\bibitem[Cor98]{Cordovil}
R.~Cordovil, \emph{The fundamental group of the complement of the
  complexification of a real arrangement of hyperplanes}, Adv. in Appl. Math.
  \textbf{21} (1998), no.~3, 481--498.

\bibitem[CS97]{CohenSuciu}
D.C. Cohen and A.I. Suciu, \emph{The braid monodromy of plane algebraic curves
  and hyperplane arrangements}, Comment. Math. Helv. \textbf{72} (1997), no.~2,
  285--315.

\bibitem[Fal88]{Falk:minimal}
M.~Falk, \emph{The minimal model of the complement of an arrangement of
  hyperplanes}, Trans. Amer. Math. Soc. \textbf{309} (1988), no.~2, 543--556.

\bibitem[FR85]{FalkRandell:LCS}
M.~Falk and R.~Randell, \emph{The lower central series of a fiber-type
  arrangement}, Invent. Math. \textbf{82} (1985), no.~1, 77--88.

\bibitem[FR87]{FalkRandell}
\bysame, \emph{On the homotopy theory of arrangements}, Complex analytic
  singularities, Adv. Stud. Pure Math., vol.~8, North-Holland, Amsterdam, 1987,
  pp.~101--124.

\bibitem[FR00]{FalkRandell2}
\bysame, \emph{On the homotopy theory of arrangements. {II}},
  Arrangements---{T}okyo 1998, Adv. Stud. Pure Math., vol.~27, Kinokuniya,
  Tokyo, 2000, pp.~93--125.

\bibitem[GAP17]{GAP4}
The GAP~Group, \emph{{GAP -- Groups, Algorithms, and Programming, Version
  4.8.8}}, 2017.

\bibitem[Gue16]{Gue:4tuple}
B.~Guerville{-Ball{\'e}}, \emph{An arithmetic {Z}ariski 4-tuple of twelve
  lines}, Geom. Topol. \textbf{20} (2016), no.~1, 537--553.

\bibitem[GV17]{GueViu:config}
B.~Guerville{-Ball{\'e}} and J.~Viu{-Sos}, \emph{Configurations of points and
  topology of real line arrangements}, Preprint available at
  \texttt{arXiv:1702.00922}, (2017).

\bibitem[Hir93]{Eriko}
E.~Hironaka, \emph{Abelian coverings of the complex projective plane branched
  along configurations of real lines}, Mem. Amer. Math. Soc. \textbf{105}
  (1993), no.~502, vi+85.

\bibitem[HN16]{nq}
M.~Horn and P.~Nickel, \emph{{Nilpotent quotients of finitely presented groups,
  {\tt nq} a {\tt GAP} package, version 2.5.3}}, 2016.

\bibitem[JY93]{JiangYau}
T.~Jiang and S.S.-T. Yau, \emph{Topological invariance of intersection lattices
  of arrangements in {$\mathbb{CP}^2$}}, Bull. Amer. Math. Soc. (N.S.)
  \textbf{29} (1993), no.~1, 88--93.

\bibitem[JY94]{JiangYau2}
\bysame, \emph{Diffeomorphic types of the complements of arrangements of
  hyperplanes}, Compositio Math. \textbf{92} (1994), no.~2, 133--155.

\bibitem[Kam33]{vanKampen}
E.R.~Van Kampen, \emph{On the {F}undamental {G}roup of an {A}lgebraic
  {C}urve}, Amer. J. Math. \textbf{55} (1933), no.~1-4, 255--260.

\bibitem[Lib86]{Libgober}
A.~Libgober, \emph{On the homotopy type of the complement to plane algebraic
  curves}, J. Reine Angew. Math. \textbf{367} (1986), 103--114.

\bibitem[Moi81]{Moishezon}
B.~G. Moishezon, \emph{Stable branch curves and braid monodromies}, Algebraic
  geometry ({C}hicago, {I}ll., 1980), Lecture Notes in Math., vol. 862,
  Springer, Berlin-New York, 1981, pp.~107--192.

\bibitem[MS02]{MateiSuciu}
D.~Matei and A.I. Suciu, \emph{Hall invariants, homology of subgroups, and
  characteristic varieties}, Int. Math. Res. Not. (2002), no.~9, 465--503.

\bibitem[MT88]{MoiTei}
B.~Moishezon and M.~Teicher, \emph{Braid group technique in complex geometry.
  {I}. {L}ine arrangements in {${\bf C}{\rm P}^2$}}, Braids ({S}anta {C}ruz,
  {CA}, 1986), Contemp. Math., vol.~78, Amer. Math. Soc., Providence, RI, 1988,
  pp.~425--555.

\bibitem[OS80]{OrlikSolomon}
P.~Orlik and L.~Solomon, \emph{Combinatorics and topology of complements of
  hyperplanes}, Invent. Math. \textbf{56} (1980), no.~2, 167--189.

\bibitem[OT92]{OrlikTerao92}
P.~Orlik and H.~Terao, \emph{Arrangements of hyperplanes}, Grundlehren der
  Mathematischen Wissenschaften, vol. 300, Springer-Verlag, Berlin, 1992.

\bibitem[Pas99]{pasq:99}
A.~Di Pasquale, \emph{Links and complements of arrangements of complex
  projective plane algebraic curves}, Ph.D. thesis, The University of
  Melbourne, 1999.

\bibitem[Ran82]{Randell}
R.~Randell, \emph{The fundamental group of the complement of a union of complex
  hyperplanes}, Invent. Math. \textbf{69} (1982), no.~1, 103--108. 
	\emph{Correction: }Invent. Math. \textbf{80} (1985), no.~3, 467--468.


\bibitem[Ran89]{Randell:lattice}
\bysame, \emph{Lattice-isotopic arrangements are topologically isomorphic},
  Proc. Amer. Math. Soc. \textbf{107} (1989), no.~2, 555--559.

\bibitem[Ryb11]{Rybnikov}
G.L. Rybnikov, \emph{On the fundamental group of the complement of a complex
  hyperplane arrangement}, Funktsional. Anal. i Prilozhen. \textbf{45} (2011),
  no.~2, 71--85, Preprint available at {\tt arXiv:math.AG/9805056}.

\bibitem[Sal88a]{Salvetti1}
M.~Salvetti, \emph{Arrangements of lines and monodromy of plane curves},
  Compositio Math. \textbf{68} (1988), no.~1, 103--122.

\bibitem[Sal88b]{Salvetti}
\bysame, \emph{On the homotopy type of the complement to an arrangement of
  lines in {${\mathbb C}^2$}}, Boll. Un. Mat. Ital. A (7) \textbf{2} (1988),
  no.~3, 337--344.

\bibitem[Suc01]{Suciu}
A.I. Suciu, \emph{Fundamental groups of line arrangements: enumerative
  aspects}, Advances in algebraic geometry motivated by physics ({L}owell,
  {MA}, 2000), Contemp. Math., vol. 276, Amer. Math. Soc., Providence, RI,
  2001, pp.~43--79.

\bibitem[Zar29]{zariski}
O.~Zariski, \emph{On the {P}roblem of {E}xistence of {A}lgebraic {F}unctions
  of {T}wo {V}ariables {P}ossessing a {G}iven {B}ranch {C}urve}, Amer. J. Math.
  \textbf{51} (1929), no.~2, 305--328.

\end{thebibliography}
\def\cprime{$'$}
\providecommand{\bysame}{\leavevmode\hbox to3em{\hrulefill}\thinspace}
\providecommand{\MR}{\relax\ifhmode\unskip\space\fi MR }
\providecommand{\MRhref}[2]{%
  \href{http://www.ams.org/mathscinet-getitem?mr=#1}{#2}
}
\providecommand{\href}[2]{#2}

\end{document}